\documentclass{nyjm}

\usepackage{amssymb}
\usepackage{mathrsfs}

\newcommand{\union}{\mathop{\bigcup}\limits}
\newcommand{\zfc}{\mathnormal{\mathsf{ZFC}}}
\newcommand{\inter}{\mathop{\bigcap}\limits}
\newcommand{\ch}{\mathnormal{\mathsf{CH}}}
\newcommand{\supp}{\mathop{\mathrm{supp}}}
\newcommand{\fd}{\mathop{\mathrm{F}\bigtriangleup}}
\newcommand{\fs}{\mathop{\mathrm{FS}}}

\newtheorem{theorem}{Theorem}[section]
\newtheorem{claim}[theorem]{Claim}
\newtheorem{lemma}[theorem]{Lemma}
\newtheorem{corollary}[theorem]{Corollary}

\theoremstyle{definition}
\newtheorem{definition}[theorem]{Definition}

\title{Every strongly summable ultrafilter on $\bigoplus\mathbb Z_2$ is sparse}
\author[D. Fern\'andez]{David J. Fern\'andez Bret\'on}
\address{Department of Mathematics and Statistics \\ York University \\ 4700 Keele Street \\ Toronto, Ontario, Canada \\ 
M3J 1P3.} 
\email{davidfb@mathstat.yorku.ca}
\urladdr{http://math.yorku.ca/\textasciitilde davidfb/}
\thanks{The author would like to thank the support received from the Consejo Nacional de Ciencia y Tecnolog\'{\i}a 
(CONACyT), Mexico, by means of 
scholarship number 213921/309058.}
\date{\today}

\begin{document}

\begin{abstract}
We investigate the possibility of the existence of nonsparse strongly summable ultrafilters on certain abelian 
groups. In particular, we show that every strongly summable ultrafilter on the countably infinite Boolean group is 
sparse. This answers a question of Hindman, Stepr\=ans and Strauss.
\end{abstract}

\maketitle
\tableofcontents

\section{Introduction}

The concept of the Stone-\v Cech compactification of a semigroup has become one of central importance, and has 
been studied extensively. Throughout this paper, we think of the Stone-\v Cech compactification of a 
discrete abelian semigroup $G$ as the set $\beta G$ of ultrafilters on $G$, 
where the point $x\in G$ is identified with the principal ultrafilter $\{A\subseteq G\big|x\in A\}$, and the basic 
open sets 
are those of the form $\bar{A}=\{p\in\beta G\big|A\in p\}$, for $A\subseteq G$. Then these sets are actually clopen, 
and 
$\bar{A}$ is really the closure in $\beta G$ of the set $A$, regarded as a subset of $\beta G$ under the 
aforementioned 
identification of points in $G$ with principal ultrafilters. The semigroup operation $+$ on $G$ is also extended by the 
formula
$$p+q=\{A\subseteq G\big|\{x\in G\big|\{y\in G\big|x+y\in A\}\in q\}\in p\}$$
which turns $\beta G$ into a right topological semigroup, meaning that for each $p\in\beta G$ the mapping 
$(\cdot)+p:\beta G\longrightarrow\beta G$ is continuous (note that the extended operation $+$ need not be 
commutative, and, even if $G$ is a group, elements $p\in G^*=\beta G\setminus G$ do not necessarily have inverses). 
The details of this 
construction (as well as a lot more information, along with applications) can be seen in \cite{hindmanstrauss}. In 
this paper, we will focus mainly on the case when $G$ is a group.

The lowercase roman letters $p,q,r$ are reserved for ultrafilters, while the uppercase roman letters $A,B,C,D$, 
with or 
without subscripts, will always denote subsets of the abelian group at hand. We will use the von Neumann natural 
numbers, 
i.e., a 
natural number $n$ is viewed as the set 
$\{0,\ldots,n-1\}$ (with $0$ equal to $\varnothing$, the empty set); and $\omega$ will denote the set of finite ordinals, 
i.e. the set of natural numbers along with zero (thus the symbols $\in$ and $<$ mean the same when applied to natural 
numbers, $0$, and to $\omega$ itself). The lowercase roman letters $i,j,k,l,m,n$, with or without subscript, will 
be reserved to denote elements of $\omega$. The lowercase roman letters $a,b,c$, with or without subscript, will stand for 
elements of $[\omega]^{<\omega}$, i.e. for finite subsets 
of $\omega$. The letters $M$ and $N$, with or without subscripts, will in general be reserved for denoting (finite or 
infinite) subsets of $\omega$. Given a subset $M\subseteq\omega$, $[M]^{<\omega}$ will denote the set of 
finite subsets of $M$, and $[M]^{\omega}$ denotes the set of infinite subsets of $M$. Of the groups that we study 
here, one of the most important ones is the circle group $\mathbb T=\mathbb R/\mathbb Z$. When dealing with 
this group, we will identify its elements (which are cosets modulo $\mathbb Z$) with their unique representative $t$ 
satisfying $0\leq t<1$. Therefore, when we refer to an element of $\mathbb T$ as a real number in $[0,1)$, we really 
mean the coset of that number modulo $\mathbb Z$.

\begin{definition}
If $G$ is an abelian semigroup, we say that an ultrafilter $p\in\beta G$ is \textbf{strongly summable} if it has a 
base of $\fs$-sets, i.e. if for every $A\in p$ 
there exists a sequence $\vec{x}=\langle x_n\big|n<\omega\rangle$ such that $p\ni\fs(\vec{x})\subseteq A$, where
$$\fs(\vec{x})=\left\{\sum_{n\in a}x_n\bigg|a\in[\omega]^{<\omega}\setminus\{\varnothing\}\right\}$$
denotes the \emph{set of finite sums of the sequence $\vec{x}$}.
\end{definition}

Note that if a strongly summable ultrafilter is principal, then it must actually be $0$. Strongly summable ultrafilters on 
$(\omega,+)$ were first constructed, under $\ch$, by Neil Hindman in \cite{hindman}, although at that time 
the terminology was still not in use. Their importance at first came from the fact that they are examples of idempotents in 
$\beta\omega$, but among idempotents they are special in that the largest subgroup of $\omega^*=\beta\omega\setminus\omega$ 
containing one of them as the identity is just a copy of $\mathbb Z$. More concretely, \cite[Th. 12.42]{hindmanstrauss} 
establishes that if $p\in\omega^*$ is a strongly summable ultrafilter, and $q,r\in\omega^*$ are such that $q+r=r+q=p$, 
then $q,r\in\mathbb Z+p$. In \cite{protasov}, the authors generalize some results previously only known to hold for 
ultrafilters on $\beta\omega$ or $\beta\mathbb Z$. In particular, they proved there that every strongly summable ultrafilter 
$p$ on any abelian group $G$ is an idempotent (\cite[Th. 2.3]{protasov}). And \cite[Th. 4.6]{protasov} states that if 
$G$ can 
be embedded 
in $\mathbb T$, then whenever $q,r\in G^*=\beta G\setminus G$ are such that $q+r=r+q=p$, 
it must be the 
case that $q,r\in G+p$. It is possible to get a slightly stronger result if one strengthens the definition of strongly 
summable.

\begin{definition}
An ultrafilter $p\in\beta G$ is \textbf{sparse} if for every $A\in p$ there exist two sequences 
$\vec{x}=\langle x_n\big|n<\omega\rangle$, $\vec{y}=\langle y_n\big|n<\omega\rangle$, where $\vec{y}$ is a subsequence 
of $\vec{x}$ such that $\{x_n\big|n<\omega\}\setminus\{y_n\big|n<\omega\}$ is infinite, $\fs(\vec{x})\subseteq A$, and 
$\fs(\vec{y})\in p$.
\end{definition}

Then obviously every sparse ultrafilter will be nonprincipal and strongly summable. And (\cite[Th. 4.5]{protasov}) 
if $G$ can be 
embedded in 
$\mathbb T$ and $p\in G^*$ is sparse, then whenever $q,r\in G^*$ are such that $q+r=p$, it must be the case that 
$q,r\in G+p$.

In \cite{jurisetal}, the authors investigate the different kinds of abelian semigroups on which 
every nonprincipal 
strongly summable ultrafilter must be sparse. For example, every nonprincipal strongly summable ultrafilter 
$p\in\omega^*$ must actually be sparse (this follows from 
\cite[Th. 3.2]{jurisetal} together with either \cite[Lemmas 12.20, 12.32]{hindmanstrauss} or  
\cite[Lemmas 1A, 1C]{blasshindman}). Thus the above result about $p$ being expressible as a sum only trivially holds 
for all nonprincipal strongly 
summable ultrafilters on $\omega$. More generally, \cite[Th. 4.2]{jurisetal} establishes that if $S$ 
is a 
countable 
subsemigroup of $\mathbb T$, then every nonprincipal strongly summable ultrafilter on $S$ is sparse. After, they 
build on this to prove 
a more general result.

\begin{theorem}[\cite{jurisetal}, Th. 4.5]
Let $S$ be a countable subsemigroup of $\bigoplus_{n<\omega}\mathbb T$ and let $p$ be a nonprincipal strongly summable 
ultrafilter on $S$. If
$$\left\{x\in S\big|\pi_{\min(x)}(x)\neq\frac{1}{2}\right\}\in p,$$
then $p$ must be sparse (here $\min(x)$ denotes the least $i$ such that $\pi_i(x)$ is nonzero).
\end{theorem}

So, for example, this theorem, as well as the method for proving it, cannot be applied if $p$ contains the 
set of $x\in\bigoplus_{n<\omega}\mathbb T$ all of whose nonzero entries equal $1/2$. This set is isomorphic to the 
countably infinite Boolean group $\bigoplus_{n<\omega}\mathbb Z_2$. While \cite{jurisetal} was still a preprint, it 
contained the question of whether it is consistent with $\zfc$ that there exists a nonprincipal nonsparse strongly 
summable ultrafilter on 
$\bigoplus_{n<\omega}\mathbb Z_2$. This question is answered in the negative in 
section 2, while section 3 gives a slight improvement of \cite[Cor. 4.6]{jurisetal}.

\section{Strongly Summable Ultrafilters in the Boolean Group}

By the Boolean group we mean the unique (up to isomorphism) countably infinite group all of whose nonidentity 
elements have order 
$2$. This group is usually thought of as the direct sum of countably many copies of $\mathbb Z_2$. However, we will think 
of it as the group whose underlying set is $[\omega]^{<\omega}$, equipped with the symmetric difference 
$\bigtriangleup$ as the group operation. For $(a_n)_{n<\omega}\in\bigoplus_{n<\omega}\mathbb Z_2$, we can define the 
support of $(a_n)_{n<\omega}$ by
$$\supp(a_n)_{n<\omega}=\{n<\omega\big|a_n=1\},$$
so that the mapping $(a_n)_{n<\omega}\longmapsto\supp(a_n)_{n<\omega}$ is an isomorphism from 
$\bigoplus_{n<\omega}\mathbb Z_2$ 
onto $[\omega]^{<\omega}$.

When dealing with $\fs$-sets on this group, we will talk about sets instead of sequences. Thus, if 
$\vec{x}=\langle x_n\big|n<\omega\rangle$ is a sequence of elements of $[\omega]^{<\omega}$, and 
$X=\{x_n\big|n<\omega\}$ is the range of that sequence, then instead of $\fs(\vec{x})$ we will write $\fd(X)$, the 
set of 
``finite symmetric differences'', in order to emphasize that the elements of our group are sets and that their 
``sum'' 
actually 
corresponds to taking symmetric differences, and using the fact that, even if the sequence $\vec{x}$ is not 
injective, that does not alter the resulting $\fs$-set. This means that, for example, if $x_i=x_j$, and $i,j\in a$ for $i\neq j$, then $\sum_{k\in a}x_k=\sum_{k\in a\setminus\{i,j\}}x_k$, due to the fact that every element of our group at hand has order $2$. We will use the uppercase roman letters $X,Y,Z$ to denote infinite subsets of $[\omega]^{<\omega}$ whenever we are interested in considering their sets of finite symmetric differences. The main result of this section, and of this 
paper, is the following theorem.

\begin{theorem}\label{teorema}
Let $p$ be a nonprincipal strongly summable ultrafilter on $[\omega]^{<\omega}$. Then, $p$ is sparse.
\end{theorem}

In order to prove this result, we need first of all a lemma which tells us that weakly summable ultrafilters in 
$[\omega]^{<\omega}$ have a property that is somewhat analogous to that of extending the Fr\'echet filter. Recall 
that an ultrafilter $p$ on an abelian semigroup $G$ is \textbf{weakly summable} if for every $A\in p$ there is a 
sequence 
$\vec{x}$ of elements of $G$ such that $\fs(\vec{x})\subseteq A$. Thus every strongly summable ultrafilter is 
weakly summable, and actually (\cite[Th. 12.17]{hindmanstrauss}) every idempotent ultrafilter on an arbitrary 
semigroup is weakly summable (and in fact \cite[Th. 12.17]{hindmanstrauss} an ultrafilter is weakly summable if and 
only if it is a 
closure point in $\beta G$ of the set of idempotents). Notice that a principal weakly summable ultrafilter must 
be idempotent, in particular if $G$ is a group then the only principal weakly summable ultrafilter is the one 
that corresponds to the identity element.

\begin{lemma}\label{keylemma}
Let $p$ be a weakly summable ultrafilter on $[\omega]^{<\omega}$. Then for any $n<\omega$, there exists an $A\in p$ 
such that $n\notin\union A$.
\end{lemma}

\begin{proof}
If $p$ is principal, then $\{\varnothing\}\in p$ will do. Otherwise, let $A_0=\{a\in[\omega]^{<\omega}\big|n\notin a\}$ and 
$A_1=[\omega]^{<\omega}\setminus A_0=\{a\in[\omega]^{<\omega}\big|n\in a\}$. There is $j\in2$ such that $A_j\in p$. 
But $j$ cannot equal $1$, for otherwise, since $p$ is weakly summable, there would be an infinite set 
$X\subseteq[\omega]^{<\omega}$ such that 
$\fd(X)\subseteq A_1$, so if $x,y\in X$ are two distinct elements, we have that $n\in x$ and $n\in y$, thus 
$n\notin x\bigtriangleup y\in\fd(X)\subseteq A_1$, a contradiction. Therefore $A_0\in p$, and certainly it is true 
that 
$n\notin\union A_0$.
\end{proof}

\begin{corollary}
If $p$ is a weakly summable ultrafilter, then for any finite subset $a$ of $\omega$, there is $A\in p$ such that 
$\union A$ is disjoint from $a$.
\end{corollary}

\begin{proof}
If $p$ is principal, then $\{\varnothing\}\in p$ will do. Otherwise, for each $n\in a$, choose $A_n\in p$ such that 
$n\notin\union A_n$. Then $p\ni A=\inter_{n\in a}A_n$, and certainly this set is as required.
\end{proof}

Originally, the author had a much more involved proof for the previous corollary, whith ideas 
similar to those of \cite[Th. 2.6]{jurisetal} and \cite[Th. 4]{krautzberger}, until he came up with the much simpler 
one that is presented above.

The fact that all elements of $[\omega]^{<\omega}$ have order $2$ has some remarkable consequences, amongst which the 
following is relevant for our purposes.

\begin{lemma}\label{fdoffd}
Let $X\subseteq[\omega]^{<\omega}$. Then, $\fd(\fd(X))=\fd(X)$.
\end{lemma}

\begin{proof}
The ``$\supseteq$" part of the equality follows from the fact that $X\subseteq\fd(X)$, and holds in any (semi)group. 
Now let us illustrate the ``$\subseteq$" part with the case when we add two finite sums. Thus let 
$a,b\in[X]^{<\omega}\setminus\{\varnothing\}$ be distinct, and notice that, since every element in our group at hand 
has order two, the following holds:
$$\sum_{x\in a}x+\sum_{y\in b}y=\sum_{z\in a\bigtriangleup b}z\in\fd(X),$$
and from this it is easy to conclude, by induction, the desired result.
\end{proof}

Now in order to prove our main result, namely Theorem~\ref{teorema}, let $p$ be a nonprincipal strongly summable 
ultrafilter on 
$[\omega]^{<\omega}$. We want to show that $p$ is sparse, thus pick $A\in p$, and pick $Z$ such that $p\ni\fd(Z)\subseteq A$. We would like to find some sets $X,Y$ such that $Y\subseteq X$, $\fd(X)\subseteq A$, $\fd(Y)\in p$ and $X\setminus Y$ is infinite. We will do so as follows.

\begin{claim}\label{claim}
It is possible to find a $Y$ such that $Y\subseteq\fd(Z)$, $\fd(Y)\in p$, and such that there are infinitely many 
$z\in Z$ with $z\notin Y$.
\end{claim}

\begin{proof}[Proof of Theorem~\ref{teorema} from Claim~\ref{claim}]
Let $X=Y\cup Z$. Then Claim~\ref{claim} 
guarantees that $X\setminus Y$ is infinite. Moreover $\fd(Y)\in p$, and now by Lemma~\ref{fdoffd} we get that 
$\fd(X)\subseteq\fd(\fd(Z))=\fd(Z)\subseteq A$, and we are done.
\end{proof}

Thus, the only thing that remains to be proved is Claim~\ref{claim}.

\begin{proof}[Proof of Claim~\ref{claim}]
Consider the set $\limsup(Z)$ which contains exactly those $n<\omega$ such that $n\in z$ for infinitely many distinct 
$z\in Z$. Then if this set is nonempty, say $n\in\limsup(Z)$, we can use Lemma~\ref{keylemma} to get a $B\in p$ such that 
$n\notin\union B$. Since $p$ is strongly summable, we can find a 
$Y$ such that $p\ni\fd(Y)\subseteq B\cap\fd(Z)$. Then 
$Y\subseteq\fd(Y)\subseteq\fd(Z)$, and for each $z\in Z$ containing $n$ (and by assumption there are infinitely many such) we 
have that $z\notin Y$, because otherwise we would have $n\in\union B$ contradicting our choice of $n$ and $B$.

The other case is when $\limsup(Z)=\varnothing$. In this case, let $M=\union Z$. Then $M$ is an infinite 
subset of $\omega$, with the property that each $n\in M$ is contained in only finitely many $z\in Z$; and we will construct 
by recursion a very special subset of $M$. Start by letting $m_0=\min(M)$, 
$Z_0=\{z\in Z\big|m_0\in z\}$ and $N_0=\union Z_0$. Then both $Z_0,N_0$ are finite and nonempty (although $Z_0$ is a 
subset of $[\omega]^{<\omega}$, whilst $N_0$ is a subset of $\omega$). Now recursively define 
$$m_{n+1}=\min\left(M\setminus\union_{k\leq n}N_k\right),$$
$Z_{n+1}=\{z\in Z\big|m_{n+1}\in z\}$, and $N_{n+1}=\union Z_{n+1}$. Then again 
$Z_{n+1}$ is finite, nonempty, and disjoint from all previous $Z_k$. Also $N_{n+1}$ is finite and nonempty, although the 
$N_k$ 
need not be disjoint, and of course $m_{n+1}>m_n$. Now notice that $Z'=\union_{k<\omega}Z_k$ is an infinite subset 
of $Z$, and if $z\in Z_k$, then 
$z\cap\{m_n\big|n<\omega\}=\{m_k\}$. Thus if we let
$N=\{m_{2n}\big|n<\omega\}$, then for every $z\in Z'$, $z\cap N$ will be nonempty if and only if $z\in Z_k$ for some 
even index $k$. 
Let $B_0=\{s\in[\omega]^{<\omega}\big|s\cap N=\varnothing\}$ and $B_1=[\omega]^{<\omega}\setminus B_0$. Now notice 
that 
whenever $z\in Z_k$ for some $k\equiv i(\!\!\!\!\mod 2)$, we must have that $z\notin B_i$. Thus, if we let $i\in 2$ be such that $B_i\in p$, we 
will have that $z\notin B_i$ for all $z\in\union_{n<\omega}Z_{2n+i}$, and there are infinitely many such. Now using the fact 
that $p$ is strongly summable, just pick $Y$ 
such that $p\ni\fd(Y)\subseteq B_i\cap\fd(Z)$. 
\end{proof}

I am thankful to Juris Stepr\=ans for pointing out an error in an earlier version of this proof, as well as to 
the anonymous referee for useful comments on it.

\section{Existence of Nonsparse Strongly Summable Ultrafilters}

In this section we will investigate a necessary condition for the existence of a nonsparse strongly summable 
ultrafilter on some abelian group $G$. This represents some partial progress towards answering \cite[Question 4.12]{jurisetal}, and sheds some light on what an answer to that question might look like.

Let $G$ be any abelian group, $S$ a subsemigroup of $G$, and $p\in\beta G$ an ultrafilter such that $S\in p$. Then 
$p\upharpoonright S=p\cap\mathfrak P(S)$ will be an ultrafilter on $S$, and it is easy to see that 
$p\upharpoonright S$ is a nonprincipal ultrafilter if and only if 
$p$ is. It is also reasonably straightforward to see that $p\upharpoonright S$ is strongly summable if and only if 
$p$ is, and also that $p\upharpoonright S$ 
is sparse if and only if $p$ is, because $A\in p$ if and only if $A\cap S\in p\upharpoonright S$. Now if $G$ is any 
infinite abelian 
group, and 
$p\in\beta G$ is a strongly summable ultrafilter, then by definition, there is a sequence 
$\vec{x}=\langle x_n\big|n<\omega\rangle$ such that $\fs(\vec{x})\in p$. If we let $S$ denote the subsemigroup of $G$ 
generated by $\{x_n\big|n<\omega\}$, then it must be the case that $S$ is countable (and cancellative). Since 
$\fs(\vec{x})\subseteq S$, then $S\in p$, and thus $q=p\upharpoonright S$ will be strongly summable. Moreover the question of 
whether $p$ is sparse reduces to the 
question of whether $q$ is sparse. Thus, when investigating the possibility of a strongly summable ultrafilter being 
nonsparse on an arbitrary abelian group, we may as well focus our attention on strongly summable ultrafilters on 
countable 
cancellative abelian semigroups.

Now if $S$ is a countable cancellative abelian semigroup, then it can be embedded in a countable abelian group $H$ 
(just in 
the same way that $(\mathbb N,+)$ can be embedded into $(\mathbb Z,+)$, or $(\mathbb Z\setminus\{0\},\cdot)$ into 
$(\mathbb Q\setminus\{0\},\cdot)$). And it is a well-known result (see, e.g., \cite[Th. 24.1]{fuchs}, \cite[4.1.6]{robinson}, or \cite[Th. 9.23]{rotman}) 
that every abelian group can be embedded in a divisible group; moreover, each divisible group is a direct sum of 
copies of $\mathbb Q$ and of quasicyclic 
groups (\cite[Th. 23.1]{fuchs}, \cite[4.1.5]{robinson}, or \cite[Th. 9.14]{rotman}). Since $\mathbb Q$, as well as 
all quasicyclic groups, can certainly 
be embedded in $\mathbb T$, the conclusion is that every countable abelian group $H$ can be embedded in the direct 
sum of 
countably many circle groups $G=\bigoplus_{n<\omega}\mathbb T$. (From now on, $G$ will denote that group). Thus we 
can think of $S$ as a subset of $G$, 
and if 
$q\in\beta S$ is an ultrafilter, then by letting $p$ be the filter on $G$ generated by $q$, we will actually get an 
ultrafilter. Moreover $S\in p$, and $q=p\upharpoonright S$, so $q$ will be strongly summable if and only if $p$ is. And again, the 
question of whether $q$ is sparse reduces to the question of whether $p$ is sparse. Therefore, the whole investigation of 
whether there is a nonsparse strongly summable ultrafilter on some abelian group (or abelian cancellative semigroup) reduces 
to the question of whether there exists a nonsparse strongly summable ultrafilter on $G$.

Our starting point will be the following Theorem of Hindman, Stepr\=ans and Strauss.

\begin{theorem}[\cite{jurisetal}, Cor. 4.6]
Let $p$ be a nonsparse nonprincipal strongly summable ultrafilter on $G$. Then $p$ contains the set of elements of 
$G$ whose order is some power of $2$.
\end{theorem}

It is not hard to see that the set of elements of $G$ whose order is a power of two is exactly 
$H=\bigoplus_{n<\omega}\mathbb T[2^\infty]$, where 
$$\mathbb T[2^\infty]=\left\{t\in\mathbb T\bigg|\left(\exists m,n\in\omega\right)\left(t=\frac{m}{2^n}\right)\right\}$$
is the quasicyclic $2$-group (also known as the Pr\"ufer group of type $2^\infty$). From now on 
we will focus on strongly summable ultrafilters on that group (which we will keep denoting by $H$). We will also be using the groups 
$$\mathbb T[2^n]=\left\{t\in\mathbb T\bigg|\left(\exists m\in\omega\right)\left(t=\frac{m}{2^n}\right)\right\}\cong\mathbb Z_{2^n},$$
the isomorphism being given by $\frac{m}{2^n}\longmapsto m$ (whenever we refer to a number $l\in\mathbb Z$ as an 
element of $\mathbb Z_k$, we really mean its coset modulo $k$). Notice that if $n<m<\omega$, then 
$\mathbb T[2^n]\subseteq\mathbb T[2^m]$, and $\mathbb T[2^\infty]=\union_{n<\omega}\mathbb T[2^n]$.

Before stating our first lemma, we need to recall a definition.

\begin{definition}[\cite{jurisetal}, Def. 3.1]
A sequence $\vec{x}$ on an abelian semigroup $S$ is said to satisfy \textbf{strong uniqueness of finite sums}
if for each $a,b\in[\omega]^{<\omega}\setminus\{\varnothing\}$, 
\begin{itemize}
 \item If $\sum_{k\in a}x_k=\sum_{k\in b}x_k$ then $a=b$.
 \item If $\sum_{k\in a}x_k+\sum_{k\in b}x_k\in\fs(\vec{x})$, then $a\cap b=\varnothing$.
\end{itemize}
\end{definition}

By \cite[Th. 3.2]{jurisetal}, if $p$ is a nonprincipal ultrafilter, and for each $A\in p$ there is a sequence 
$\vec{x}$ satisfying strong uniqueness of finite sums such that $p\ni\fs(\vec{x})\subseteq A$, then $p$ is sparse.

\begin{lemma}\label{elunmedio}
Let $p$ be a nonsparse strongly summable ultrafilter on $H$. Then for each $n<\omega$, the set 
$B_n=\pi_n^{-1}[\mathbb T[2]]=\{x\in H\big|\pi_n(x)\in\{0,1/2\}\}\in p$.
\end{lemma}

\begin{proof}
We proceed by contraposition, so let us assume that there is $n<\omega$ such that $B_n\notin p$, and, essentially 
without 
loss of 
generality, let us also assume that $\{x\in H\big|\pi_n(x)\in(0,1/2)\}\in p$. Pick $j\in3$ such that $X_j\in p$, where
$$X_j=\left\{x\in H\bigg|\pi_n(x)\in\union_{m<\omega}\left[\frac{1}{2^{3m+j+2}},\frac{1}{2^{3m+j+1}}\right)\right\},$$
i.e., thinking of $\pi_n(x)$ as a number in $(0,1/2)$ written in binary notation, its first digits will be $0.0$ 
and then there will be an infinite string of zeroes and ones. Then $x\in X_j$ if and only if the first such nonzero 
digit 
appears in a position that is congruent with $j$ modulo $3$. Let $C\in p$, and let $\vec{x}$ be a sequence of 
elements of $H$ satisfying $p\ni\fs(\vec{x})\subseteq C\cap X_j$. Note that if $l\neq k$ and for some $m$, we have 
that $\frac{1}{2^{3m+j+2}}\leq \pi_n(x_l)<\frac{1}{2^{3m+j+1}}$ and 
$\frac{1}{2^{3m+j+2}}\leq \pi_n(x_k)<\frac{1}{2^{3m+j+1}}$, then 
$\frac{1}{2^{3m+j+1}}\leq \pi_n(x_l+x_k)<\frac{1}{2^{3m+j}}$ and so $x_l+x_k\notin X_j$, which is impossible. 
Thus there is at most one $\pi_n(x_k)$ in each interval $\left[\frac{1}{2^{3m+j+2}},\frac{1}{2^{3m+j+1}}\right)$ (the 
positions of the first nonzero digits of 
distinct $\pi_n(x_k)$ are distinct), so we may assume that the sequence $\vec{x}$ is arranged in such a way that 
$k<l$ 
implies $\pi_n(x_k)>\pi_n(x_l)$ ($\vec{x}$ is arranged in increasing order of its first nonzero digit in the $n$-th projection). Consequently, for each $k<\omega$ we have that $\pi_n(x_k)>4\pi_n(x_{k+1})$ and therefore 
$\pi_n(x_k)>3\sum_{l=k+1}^\infty\pi_n(x_l)$. This is easily 
seen to imply that for $a,b\in[\omega]^{<\omega}\setminus\{\varnothing\}$ and for 
$\varepsilon:a\longrightarrow\{1,2\},\delta:b\longrightarrow\{1,2\}$, if 
$\sum_{k\in a}\varepsilon(k)\pi_n(x_k)=\sum_{k\in b}\delta(k)\pi_n(x_k)$ then $a=b$ and $\varepsilon=\delta$. And of 
course this implies that for $a,b\in[\omega]^{<\omega}\setminus\{\varnothing\}$ and for 
$\varepsilon:a\longrightarrow\{1,2\},\delta:b\longrightarrow\{1,2\}$, if 
$\sum_{k\in a}\varepsilon(k)x_k=\sum_{k\in b}\delta(k)x_k$ then $a=b$ and $\varepsilon=\delta$. The latter statement 
in turn easily implies that the sequence $\vec{x}$ satisfies strong uniqueness of finite sums, hence, by 
\cite[Th. 3.2]{jurisetal}, $p$ must be sparse.
\end{proof}

The following lemma is stated in more generality than will actually be needed. Notice that we can recover 
Lemma~\ref{keylemma} as a particular case of it.

\begin{lemma}\label{finitecoordinate}
Let $p$ be a weakly summable ultrafilter on $H$, and assume that for some $n<\omega$ there is an $A\in p$ such that $\pi_n[A]$ is finite. Then $\{x\in G\big|\pi_n(x)=0\}\in p$.
\end{lemma}

\begin{proof}
Enumerate the finite set $\pi_n[A]=\{g_0,\ldots,g_{k-1}\}$ and choose $i<k$ such that 
$A_i=\{x\in A\big|\pi_n(x)=g_i\}\in p$. Since $p$ is weakly summable, we can pick a sequence $\vec{x}$ of elements 
of $G$ such that $\fs(\vec{x})\subseteq A_i$. But then, for example, $x_0,x_1,x_0+x_1\in A_i$, thus $g_i=\pi_n(x_0+x_1)=\pi_n(x_0)+\pi_n(x_1)=g_i+g_i$ and this implies that $g_i=0$.
\end{proof}

\begin{corollary}\label{seconcentraencero}
Let $p$ be a nonsparse strongly summable ultrafilter on $H$. Then for each $n<\omega$, $\{x\in H\big|\pi_n(x)=0\}\in p$.
\end{corollary}

\begin{proof}
Just put together Lemmas~\ref{elunmedio} and \ref{finitecoordinate}.
\end{proof}

In what follows we will use \cite[Th. 4.5]{jurisetal}, which says that if $p$ is a nonprincipal, strongly summable 
ultrafilter on a subsemigroup $S$ of $G=\bigoplus_{n<\omega}\mathbb T$, and if 
$$\{x\in S\setminus\{0\}\big|\pi_{\min(x)}(x)\neq1/2\}\in p$$
(where $\min(x)$ denotes the least $n$ such that $\pi_n(x)\neq0$), then there exists an $X\in p$ such that any 
sequence 
$\vec{x}$ with $\fs(\vec{x})\subseteq X$ satisfies strong uniqueness of finite sums (and in particular $p$ is 
sparse). Now assume that $p$ is a nonprincipal, nonsparse strongly summable ultrafilter on $H$. Notice that $p$ 
cannot contain the set $\{x\in G\big|(\forall n<\omega)(\pi_n(x)\in\{0,1/2\})\}=\bigoplus_{n<\omega}\mathbb T[2]$, 
because this set is a copy of 
$\bigoplus_{n<\omega}\mathbb Z_2$ and hence if $p$ contains it, that would induce a strongly summable ultrafilter 
$q$ on the Boolean group, which by Theorem~\ref{teorema} must be sparse and therefore $p$ will also be sparse. Hence 
$p$ must contain the set 
$C=\{x\in G\big|(\exists n<\omega)(\pi_n(x)\notin\{0,1/2\})\}$. For $x\in C$, let $\rho(x)$ denote the least $n$ 
such that $\pi_n(x)\notin\{0,1/2\}$.

\begin{lemma}\label{eluncuarto}
Let $p$ be a strongly summable ultrafilter on $H$. If
$$\{x\in C\big|\pi_{\rho(x)}(x)\notin\{1/4,3/4\}\}\in p$$
then $p$ is sparse.
\end{lemma}

\begin{proof}
Consider the morphism $\varphi:H\longrightarrow H$ given by $\varphi(x)=2x$, whose kernel is exactly 
$\bigoplus_{n<\omega}\mathbb T[2]$. Since the latter is not an element of $p$, then 
$\varphi(p)=(\beta\varphi)(p)$ (i.e. the image of $p$ under the continuous extension of 
$\varphi:H\longrightarrow\beta H$ to $\beta H$, which is given by $\{A\subseteq H\big|\varphi^{-1}[A]\in p\}$) is 
a nonprincipal ultrafilter. Moreover, since $p$ is strongly summable, by \cite[Lemma 4.4]{jurisetal}, so is 
$\varphi(p)$. Now notice that for $x\in H\setminus\ker(\varphi)=C$, we have $\rho(x)=\min(\varphi(x))$. Thus 
$\varphi(p)$ contains the set $\{x\in H\setminus\{0\}\big|\pi_{\min(x)}(x)\neq1/2\}$, since its preimage under 
$\varphi$ is exactly $\{x\in C\big|\pi_{\rho(x)}(x)\notin\{1/4,3/4\}\}$. Therefore by \cite[Th. 4.5]{jurisetal}, 
there is a 
set $X\in\varphi(p)$ such that whenever $\fs(\vec{y})\subseteq X$, $y$ must satisfy strong uniqueness of finite sums. 
Now for $A\in p$, we can pick a sequence $\vec{x}$ such that $p\ni\fs(\vec{x})\subseteq A\cap\varphi^{-1}[X]$. 
Then if we let $\vec{y}$ be the sequence given by $y_n=\varphi(x_n)$, we get that 
$\fs(\vec{y})=\varphi[\fs(\vec{x})]\subseteq X$, thus $\vec{y}$ satisfies strong uniqueness of finite sums. 
It is not hard to see that this implies that $\vec{x}$ satisfies strong uniqueness of finite sums as well, 
thus $p$ has a basis of sets of the form $\fs(\vec{x})$ for sequences $\vec{x}$ satisfying strong uniqueness 
of finite sums. Therefore by \cite[Th. 3.2]{jurisetal}, $p$ is sparse.
\end{proof}

Now we are ready to state the main result of this section.

\begin{theorem}\label{sucesiondezetasdosenes}
Assume that there exists a nonsparse strongly summable ultrafilter $p$ on $H$. Then there exists a sequence 
$\vec{n}=\langle n_i\big|i<\omega\rangle$ of natural numbers such that 
$\bigoplus_{n<\omega}\mathbb T[2^{n_i}]\in p$. In particular, if there exists a (nonprincipal) nonsparse strongly 
summable ultrafilter on some abelian cancellative semigroup, then there exists one on 
$\bigoplus_{n<\omega}\mathbb Z_{2^{n_i}}$, for some sequence $\vec{n}$.
\end{theorem}

\begin{proof}
Let $p$ be a nonprincipal, nonsparse strongly summable ultrafilter on $H$. As was pointed out above, $p$ cannot 
contain the set $\{x\in G\big|(\forall n<\omega)(\pi_n(x)\in\{0,1/2\})\}=\bigoplus_{n<\omega}\mathbb T[2]$, hence 
$C=\{x\in G\big|(\exists n<\omega)(\pi_n(x)\notin\{0,1/2\})\}\in p$. Moreover by Lemma~\ref{eluncuarto}, $C_0=\{x\in C\big|\pi_{\rho(x)}(x)\in\{1/4,3/4\}\}\in p$. Now 
$C_0=C_1\cup C_3$, where $C_i=\{x\in C_0\big|\pi_{\rho(x)}(x)=i/4\}$. Essentially without loss of generality, we 
can assume that $C_1\in p$. Now choose a sequence $\vec{x}$ with $p\ni\fs(\vec{x})\subseteq C_1$, and for $i<\omega$ 
let $M_i=\{n<\omega\big|\rho(x_n)=i\}$ (so $n\in M_i$ implies $\pi_i(x_n)=1/4$).

\begin{claim}\label{alomas2}
For each $i<\omega$, $|M_i|\leq2$.
\end{claim}

\begin{proof}[Proof of Claim]
Assume, by way of contradiction, that there are pairwise distinct $n,m,k\in M_i$. Let $x=x_n+x_m+x_k$. For $j<i$, we 
have that $\pi_j(x)\in\{0,1/2\}$, because $\pi_j(x_n),\pi_j(x_m),\pi_j(x_k)\in\{0,1/2\}$. On the other hand, 
$\pi_i(x_n)=\pi_i(x_m)=\pi_i(x_k)=1/4$ thus $\pi_i(x)=3/4$, so $\rho(x)=i$ and $x\in C_3$, which is a contradiction.
\end{proof}

Thus, by rearranging the sequence if necessary, we may assume that $i<j$ and $n\in M_i,m\in M_j$ implies that 
$n<m$. Equivalently, $n<m$ implies that $\rho(x_n)\leq\rho(x_m)$, where the inequality is strict if $m>n+1$.

\begin{claim}\label{iszeroonpreviousrho}
Let $n<m<\omega$ and assume that $i=\rho(x_n)<\rho(x_m)$ (which may or may not hold if $m=n+1$, but must hold if 
$m>n+1$). Then $\pi_i(x_m)=0$.
\end{claim}

\begin{proof}[Proof of Claim]
Let $x=x_n+x_m$. For $j<i$, we since $\pi_j(x_n),\pi_j(x_m)\in\{0,1/2\}$ we have that $\pi_j(x)\in\{0,1/2\}$. On the 
other hand, $\pi_i(x_n)=1/4$ while $\pi_i(x_m)\in\{0,1/2\}$, so $\pi_i(x)\in\{1/4,3/4\}$. Hence $\rho(x)=i$, now 
since $x\in C_1$, whe must have $\pi_i(x)=1/4$, which can only happen if $\pi_i(x_m)=0$.
\end{proof}

\begin{claim}\label{finitesetofgenerators}
For every $i<\omega$, the set $\{\pi_i(x_n)\big|n<\omega\}$ is finite.
\end{claim}

\begin{proof}[Proof of Claim]
Let $i<\omega$.  We have two cases according to whether $M_i$ is nonempty or not.

If $M_i\neq\varnothing$, then by Claim~\ref{alomas2}, we know that $|M_i|\leq2$. Thus we can let $k=\min(M_i)$ and 
$k'=\max(M_i)$ (so that $k'$ equals either $k$ or $k+1$, $M_i=\{k,k'\}$, and $\pi_i(x_k)=\pi_i(x_{k'})=1/4$). 
Now Claim~\ref{iszeroonpreviousrho} yields $\pi_i(x_n)=0$ for $n>k'$, therefore
\begin{eqnarray*}
\{\pi_i(x_n)\big|n<\omega\} & = & \{\pi_i(x_n)\big|n<k\}\cup\{\pi_i(x_n)\big|n\in\{k,k'\}\}\cup\{\pi_i(x_n)\big|n>k\} \\
 & = & \{\pi_i(x_n)\big|n<k\}\cup\{1/4\}\cup\{0\}
\end{eqnarray*}
which is finite.

Now if $M_i=\varnothing$, then Claim~\ref{alomas2} guarantees that there are only finitely many 
integers $l$ such that $\rho(x_l)<i$, so let $k$ be the greatest such integer if some exists, or $k=0$ otherwise 
(equivalently $k=\max(M_j)$ where $j$ is the greatest integer less that $i$ for which $M_j\neq\varnothing$, if such 
a $j$ exists, or $k=0$ otherwise). Thus we know that for $n>k$, $\pi_i(x_n)\in\{0,1/2\}$. Therefore 
$$\{\pi_i(x_n)\big|n<\omega\}=\{\pi_i(x_n)\big|n\leq k\}\cup\{\pi_i(x_n)\big|n>k\}\subseteq\{\pi_i(x_n)\big|n\leq k\}\cup\{0,1/2\}$$
which is finite as well.
\end{proof}

Notice that, if $F$ is a finite subset of $\mathbb T[2^\infty]$, then the subgroup of the latter generated by the 
former must be $\mathbb T[2^n]$ for suitable $n$. Namely, if $F=\{\frac{n_0}{2^{m_0}},\ldots,\frac{n_k}{2^{m_k}}\}$, 
where $2\nmid n_i$ for $i\leq k$, and $m=\max\{m_0,\ldots,m_k\}$, then $F$ generates the subgroup $\mathbb T[2^m]$.
Hence by Claim~\ref{finitesetofgenerators}, for each $i<\omega$ we can choose $n_i\in\mathbb N$ such that the 
subgroup of $\mathbb T[2^\infty]$ generated by $\{\pi_i(x_n)\big|n<\omega\}$ is $\mathbb T[2^{n_i}]$. In this way 
we construct the sequence $\vec{n}=\langle n_i\big|i<\omega\rangle$ of natural numbers which satisfies that 
$p\ni\fs(x)\subseteq\bigoplus_{n<\omega}\mathbb T[2^{n_i}]$.
\end{proof}

Recall that $\mathbb T[2^n]\cong\mathbb Z_{2^n}$; and that, if $n<m$, then $\mathbb T[2^n]\subseteq\mathbb T[2^m]$. 
From this, it is not hard to see that the sequence $\vec{n}$ from the previous theorem has to be unbounded. For if 
that sequence was bounded, say by $n$, then we would have that 
$\bigoplus_{i<\omega}\mathbb T[2^{n_i}]\subseteq\bigoplus_{i<\omega}\mathbb T[2^n]$. Thus if $p$ is the ultrafilter 
yielding $\vec{n}$, $p$ would contain $\bigoplus_{i<\omega}\mathbb T[2^n]$, thus inducing a (nonprincipal) nonsparse 
strongly summable ultrafilter $q$ on its isomorphic copy $\bigoplus_{i<\omega}\mathbb Z_{2^n}$. But 
this cannot happen, more generally, for every $n\geq2$ every strongly summable ultrafilter $q$ on 
$G(n)=\bigoplus_{i<\omega}\mathbb Z_n$ is sparse. The case when $n=2$ is just Theorem~\ref{teorema}, and for 
$n\geq3$, pick $0<k<n$ such that $A_k\in p$, where $A_k$ is the set consisting of those $x\in G(n)$ whose first 
nonzero coordinate equals $k$. It is then easy to see that, for $A\in p$, if $\vec{x}$ is a sequence of 
elements of $G(n)$ such that $\fs(\vec{x})\subseteq A\cap A_k$, then for distinct $i,j$ the indices of the first 
nonzero coordinates of $x_i$ and $x_j$ must be different (otherwise the first nonzero coordinate of $x_i+x_j$ would 
be $2k\neq k$, so we would have $x_i+x_j\notin A_k$, which is absurd). This in turn implies that the sequence 
$\vec{x}$ satisfies the strong uniqueness of finite sums, and thus by \cite[Th. 3.2]{jurisetal} the desired 
conclusion follows.

Therefore, since every sequence $\vec{n}$ given by the theorem must be unbounded, one might be tempted to think that 
every such sequence should tend to infinity very quickly, but this is really not the case, as the following 
corollary shows. I am thankful to Andreas Blass for pointing this out to me.

\begin{corollary}
If $p$ is a nonprincipal nonsparse strongly summable ultrafilter on $H$, then there is an injective homomorphism 
$\varphi:H\longrightarrow H$ sending $p$ to an ultrafilter $q$ (which must necessarily be also nonprincipal 
nonsparse strongly summable) containing the set $\bigoplus_{n<\omega}\mathbb T[2^{n}]$. In particular, if there is a 
nonprincipal nonsparse strongly summable ultrafilter on some abelian cancellative semigroup, then there is one on 
$\bigoplus_{n<\omega}\mathbb Z_{2^n}$.
\end{corollary}

\begin{proof}
Given the sequence $\vec{n}$ from Theorem~\ref{sucesiondezetasdosenes}, create a new strictly increasing sequence 
$\vec{m}=\langle m_i\big|i<\omega\rangle$ by letting $m_0$ be the least $k$ such that $n_0<k$, and recursively 
letting $m_{i+1}$ be the least $k$ such that $\max\{n_{i+1},m_i\}<k$. Then we can define the embedding 
$\varphi:H\longrightarrow H$ by letting $\varphi(x)$ be the element of $H$ whose $m_i$-th coordinate is exactly the 
$i$-th coordinate of $x$ and whose $k$-th coordinate is zero whenever $k\notin\{m_i\big|i<\omega\}$. Clearly 
$\varphi$ is injective. Now notice that for 
$x\in\bigoplus_{i<\omega}\mathbb T[2^{n_i}]$, since by construction $n_i<m_i$, we have for every $i<\omega$ that
$\pi_{m_i}(\varphi(x))=\pi_i(x)\in\mathbb T[2^{n_i}]\subseteq\mathbb T[2^{m_i}]$, and of course for 
$k\notin\{m_i\big|i<\omega\}$ we have that $\pi_k(\varphi(x))=0\in\mathbb T[2^k]$. Thus 
$\varphi(x)\in\bigoplus_{n<\omega}\mathbb T[2^n]$, hence 
$\varphi\left[\bigoplus_{i<\omega}\mathbb T[2^{n_i}]\right]\subseteq\bigoplus_{n<\omega}\mathbb T[2^n]$ and so the latter 
is an element of $\varphi(p)$, and thus the result follows.
\end{proof}


\begin{thebibliography}{9}

\bibitem{blasshindman} {\scshape Blass, A.; Hindman, N.} On Strongly Summable Ultrafilters and Union Ultrafilters, 
{\em Trans. Amer. Math. Soc.} {\bf 304} No. 1 (1987), 83-99. MR0906807 (88i:03080), Zbl 0643.03032

\bibitem{fuchs} {\scshape Fuchs, L.} Infinite Abelian Groups, vol. I. Pure and Applied Mathematics 36. {\em Academic 
Press, New York-San Francisco-London}, 1970. MR0255673 (41 \#333), Zbl 0209.05503.

\bibitem{hindman} {\scshape Hindman, N.} The existence of certain ultrafilters on $\mathbb N$ and a conjecture of 
Graham and Rothschild, {\em Proc. Amer. Math. Soc.} \textbf{36} (1972), 341-346. MR0307926 (46 \#7041), 
Zbl 0259.10046.

\bibitem{protasov} {\scshape Hindman, N.; Protasov, I.; Strauss, D.} Strongly Summable Ultrafilters on Abelian 
Groups, {\em Matem. Studii} \textbf{10} (1998), 121-132. MR1687143 (2001d:22003), Zbl 0934.22005.

\bibitem{jurisetal} {\scshape Hindman, N.; Stepr\=ans, J.; Strauss, D.} Semigroups in which all Strongly Summable 
Ultrafilters are Sparse, {\em New York J. Math.} {\bf 18} (2012), 835-848. Zbl pre06098874.

\bibitem{hindmanstrauss} {\scshape Hindman, N.; Strauss, D.} Algebra in the Stone-\v Cech Compactification. 
de Gruyter Expositions in Mathematics 27. {\em Walter de Gruyter, Berlin-New York}, 1998. 
MR1642231 (99j:54001), Zbl 0918.22001.

\bibitem{krautzberger} {\scshape Krautzberger, P.} On strongly summable ultrafilters, {\em New York J. Math.} 
{\bf 16} (2010), 629-649. MR2740593 (2012k:03135), Zbl 1234.03034.

\bibitem{robinson} {\scshape Robinson, D. J. S.} A Course in the Theory of Groups. Graduate Texts in Mathematics 80. 
{\em Springer-Verlag, New York-Heidelberg-Berlin}, 1982. MR0648604 (84k:20001), Zbl 0483.20001.

\bibitem{rotman} {\scshape Rotman, J.} The Theory of Groups. {\em Allyn and Bacon, Boston-London-Sydney-Toronto}, 
1973.  MR0442063 (56 \#451)/MR0690593 (50 \#2315), Zbl 0262.20001.

\end{thebibliography}
\end{document}